\documentclass[preprint,11pt]{elsarticle}

\usepackage{amssymb}
\usepackage{amsthm}
\usepackage{tikz}
\usetikzlibrary{arrows}
\usepackage{mdframed}
\usepackage{tikz,pgfplots}
\usetikzlibrary{decorations.markings}

\usepackage[margin=2.5cm]{geometry}

\journal{Discrete Mathematics}

\begin{document}

\newtheorem{definition}{Definition}
\newtheorem{theorem}{Theorem}
\newtheorem{corollary}{Corollary}
\newtheorem{lemma}{Lemma}
\newtheorem{conjecture}{Conjecture}

\tikzset{middlearrow/.style={
        decoration={markings,
            mark= at position 0.7 with {\arrow[scale=2]{#1}} ,
        },
        postaction={decorate}
    }
}

\begin{frontmatter}


\title{Digraphs with degree two and excess two are diregular}



\author{James Tuite}

\address{Department of Mathematics and Statistics, Open University, Walton Hall, Milton Keynes}

\ead{james.tuite@open.ac.uk}

\begin{abstract}
A $k$-geodetic digraph with minimum out-degree $d$ has excess $\epsilon $ if it has order $M(d,k) + \epsilon $, where $M(d,k)$ represents the Moore bound for out-degree $d$ and diameter $k$.  For given $\epsilon $, it is simple to show that any such digraph must be out-regular with degree $d$ for sufficiently large $d$ and $k$.  However, proving in-regularity is in general non-trivial.  It has recently been shown that any digraph with excess $\epsilon = 1$ must be diregular.  In this paper we prove that digraphs with minimum out-degree $d = 2$ and excess $\epsilon = 2$ are diregular for $k \geq 2$. 

\end{abstract}


\end{frontmatter}

\section{Introduction}
\label{S:1}

The directed degree/diameter problem asks: what is the maximum possible order $n_{d,k}$ of a digraph with given maximum out-degree $d$ and diameter $k$?  Numerous applications arise in the design of large interconnection networks.  Fixing a vertex $v$, it is simple to show by induction that for $0 \leq t \leq k$ there are at most $d^t$ vertices at distance $t$ from $v$.  We therefore obtain the so-called directed Moore bound

\[ n_{d,k} \leq 1 + d + d^2 + \dots + d^k = M(d,k).\]

A digraph that attains this upper bound is called a \textit{Moore digraph}.  It is easily seen that a digraph is Moore if and only if it is out-regular with degree $d$, has diameter $k$ and is $k$-geodetic, i.e. for any two vertices $u, v$ there is at most one walk from $u$ to $v$ with length $\leq k$.  It was shown in 1980 in \cite{BriTou} that Moore digraphs exist only in the trivial cases $d = 1$ and $k = 1$ (the digraphs in question are directed cycles and complete digraphs respectively).  It is therefore of interest to find digraphs that in some sense approximate Moore digraphs.

Relaxing the requirement of $k$-geodecity, many authors have considered the problem of finding digraphs with maximum out-degree $\leq d$, diameter $\leq k$ and order $M(d,k) - \delta $ for some small \textit{defect} $\delta $.  For diameter $k = 2$, it has been shown in \cite{FioYebAle} that there exists a digraph with defect $\delta = 1$ (or an \textit{almost Moore digraph}) for every value of $d$.  However, it is known that there are no almost Moore digraphs for diameters 3 and 4 and $d \geq 2$ \cite{ConGimGonMirMor,ConGimGonMirMor2} or for degrees 2 or 3 for $k \geq 3$ \cite{BasMilSirSut,MilFri}.  It was further shown in \cite{MilSir} that no digraphs with degree $d = 2$ and defect $\delta = 2$ exist for $k \geq 3$.  The reader is referred to the survey \cite{MilSir2} for more information.

An important first step in non-existence proofs for digraphs with fixed maximum out-degree and order close to the Moore bound is to show that any such digraph must be diregular.  In \cite{MilGimSirSla}, it is shown that any almost Moore digraph must be diregular and in \cite{MilSla} it is proven that digraphs with out-degree $d = 2$ and defect $\delta = 2$ are diregular.  Further results are given in \cite{SilThesis,SlaBasMil}. 

If we preserve the $k$-geodecity requirement in the conditions for a digraph to be Moore, but instead relax the condition that the digraph should have diameter $k$, then we obtain the following interesting problem: what is the smallest possible order of a $k$-geodetic digraph with minimum out-degree $\geq d$?  We shall say that a $k$-geodetic digraph with minimum out-degree $\geq d$ and order $M(d,k) + \epsilon $ has \textit{excess} $\epsilon $ and will refer to such a digraph as a $(d,k,+\epsilon )$-digraph.  In a digraph with excess $\epsilon $, we can associate with every vertex $u$ a set of vertices $O(u)$ such that there is a path of length $\leq k$ from $u$ to $v$ if and only if $v \not \in O(u)$; this set is called the \textit{outlier set} of $u$.  If the digraph is out-regular, then every outlier set evidently has order $\epsilon $.  For excess $\epsilon = 1$, we can instead think of an outlier function $o$; $o$ is a digraph automorphism if and only if the $(d,k,+\epsilon )$-digraph is diregular.

Compared with the abundance of literature on digraphs with small defect, coverage of this problem is sparse, with \cite{Sil} an outstanding exception.  In \cite{Sil} it was proven that there are no diregular digraphs with degree $d = 2$ and excess $\epsilon = 1$.  Strong conditions were also obtained on non-diregular digraphs with excess $\epsilon = 1$ and in consequence it was shown that no such digraphs exist for $d = 2, k = 2$.  The proof of the non-existence of non-diregular digraphs with excess $\epsilon = 1$ was completed by Miller, Miret and Sillasen \cite{MirSil}.

In the present work, we demonstrate that there are no non-diregular $k$-geodetic digraphs with degree $d = 2$ and excess $\epsilon = 2$  for $k \geq 2$ in a manner analogous to \cite{MilSla} and \cite{SilThesis}.

\section{Basic lemmas for non-diregular digraphs with small excess}

Firstly, we establish our notation.  In this paper $G$ will stand for a non-diregular $k$-geodetic digraph with minimum out-degree $\geq d$ and excess $\epsilon $, i.e. a $(d,k,+\epsilon )$-digraph.  We denote the Moore bound for out-degree $d$ and diameter $k$ by $M(d,k)$ and for convenience we set $M(d,k) = 0$ for $k < 0$.  For any vertex $u$ we will write $O(u)$ for the set of $\epsilon $ outliers of $u$; we can extend $O$ to a set function by setting $O(X) = \cup_{x \in X}O(x)$ for all $X \subseteq V(G)$.  We will denote a general outlier set by $\Omega $ and will occasionally say that an outlier set is an $\Omega $-set.  We also write $O^-(u) = \{ v \in V(G) : u \in O(v)\} $ for the set of vertices of which $u \in V(G)$ is an outlier.  

The distance $d(u,v)$ between vertices $u$ and $v$ is the length of a shortest directed path from $u$ to $v$; note that in a digraph we can have $d(u,v) \not = d(v,u)$.  A path (cycle) of length $r$ will be referred to as an $r$-path (-cycle) and a $\leq r$-path (cycle) is a path (cycle) of length $\leq r$.  For $l \geq 0$ and any vertex $u$, let $N^l(u)$ be the set of vertices $v$ with an $l$-path from $u$ to $v$; similarly, for $l < 0$, $N^l(u)$ is the set of vertices $v$ that are the initial vertices of $|l|$-paths that terminate at $u$.  When $l = 1$ or $-1$, we will instead write $N^+(u)$ and $N^-(u)$ respectively.  By extension we set $N^+(X) = \cup _{x \in X}N^+(x)$ for any $X \subseteq V(G)$, with $N^-(X)$ defined analogously.  Let $T_l(u) = \cup_{j=0}^{j=l}N^j(u)$ for $l \geq 0$ and $T_l(u) = \cup_{j=l}^{j=0}N^j(u)$ for $l < 0$; for $l = k-1$ or $-(k-1)$, we put $T(u) = T_{k-1}(u)$ and $T^-(u) = T_{-(k-1)}(u)$ for short.  

The in- and out-degrees of a vertex $u$ are defined to be $d^-(u) = |N^-(u)|$ and $d^+(u) = |N^+(u)|$ respectively.  If all vertices of a digraph $G$ have the same out-degree, then $G$ is \textit{out-regular}.  If there exists $d$ such that $d^-(u) = d^+(u) = d$ for every vertex $u \in V(G)$, then $G$ is \textit{diregular}.  The sequence formed by arranging the in-degrees of the vertices of $G$ in non-decreasing order is the \textit{in-degree sequence} of $G$.  Following the notation of \cite{MilGimSirSla,Sil,MilSla}, let $S = \{ v \in V(G): d^-(v) < d\} $ and $S' = \{ v' \in V(G): d^-(v') > d\}$.    

As a preparation for the derivation of our main result, we will begin by deducing some fundamental structural results that apply for any small excess.  In the following, it will be helpful for our analysis to assume out-regularity.  The following lemma from \cite{Sil} shows that this assumption is valid for sufficiently large $d$ and $k$.  

\begin{lemma}\label{out-regular}
If $\epsilon < M(d,k-1)$, then $G$ is out-regular with degree $d$.
\end{lemma}
\begin{proof}
If $G$ is not out-regular, it must contain a vertex $v$ with out-degree at least $d+1$.  By $k$-geodecity, it follows that

\[ |V(G)| \geq |T_k(v)| \geq 1 +(d+1)+(d+1)d+ \dots +(d+1)d^{k-1} = M(d,k)+M(d,k-1). \]
As $G$ has order $M(d,k) + \epsilon < M(d,k)+M(d,k-1)$, this is a contradiction.  
\end{proof}
The conditions in Lemma \ref{out-regular} will be satisfied by all digraphs of interest in this paper.  We now present the two main lemmas on which the present work is based.  The second is a generalisation of Lemma 2.2 of \cite{Sil}.

\begin{lemma}\label{every v an outlier}
$S \subseteq \cap_{u \in V(G)}O(N^+(u))$.
\end{lemma}
\begin{proof}
Let $v \in S$ and $u \in V(G)$.  Write $N^+(u) = \{ u_1,u_2,\dots ,u_d\} $ and suppose that $v \not \in O(u_i)$ for $1 \leq i \leq d$.  Let $v \not \in N^+(u)$.  Then for $1 \leq i \leq d$ there is a $\leq k$-path from $u_i$ to $v$ and so for $1 \leq i \leq d$ there is a  $\leq (k-1)$-path from $u_i$ to $N^-(v)$.  As $d^-(v) \leq d-1$ it follows by the Pigeonhole Principle that there exists an in-neighbour $v^*$ of $v$ with two $\leq k$-paths from $u$ to $v^*$, contradicting $k$-geodecity.  Only trivial changes are necessary to deal with the case $v \in N^+(u)$. 
\end{proof}

\begin{lemma}\label{every v' neighbour of an outlier}
$S' \subseteq \cap_{u \in V(G)} N^+(O(u))$.
\end{lemma}
\begin{proof}
Let $v' \in S'$ and $u \in V(G)$.  Suppose for a contradiction that $v' \not \in N^+(O(u))$.  Then every in-neighbour of $v'$ is reachable by a $\leq k$-path from $u$.  If $u \not \in N^-(v')$, then by the Pigeonhole Principle there must exist an out-neighbour $u^*$ of $u$ with two $\leq (k-1)$-paths to $N^-(v')$, so that there are two $\leq k$-paths from $u^*$ to $v'$, a contradiction.  The result follows similarly if $u \in N^-(v')$.
\end{proof}

As every vertex has exactly $\epsilon $ outliers, this provides us with a bound on the size of the sets $S$ and $S'$.

\begin{corollary}\label{bound on |S|, |S'|}
$|S|,|S'| \leq \epsilon d$.
\end{corollary}

There are also natural restrictions on the in-degrees of vertices in $S$ and $S'$. 

\begin{lemma}\label{limit on out-degree}
For every vertex $v' \in S'$ we have $d+1 \leq d^-(v') \leq d+\epsilon $.
\end{lemma}
\begin{proof}
Let $v' \in S'$ and consider the breadth-first search tree of depth $k$ rooted at $v'$.  Write $N^+(v') = \{ v_1',v_2',\dots ,v_d'\} $.  Every in-neighbour of $v'$ lies in $(\cup_{i=1}^{d}T(v_i')) \cup O(v')$.  By $k$-geodecity, at most one in-neighbour of $v'$ lies in any set $T(v_i')$.  As there are $d$ such sets and $\epsilon $ vertices in $O(v')$, the result follows.  
\end{proof}

\begin{lemma}\label{average in-degree}
$\sum_{v \in S}(d-d^-(v)) = \sum_{v' \in S'}(d^-(v')-d)$.
\end{lemma}
\begin{proof}
By Lemma \ref{out-regular}, the average in-degree must be $d$.
\end{proof}

\begin{lemma}\label{outliers are in-neighbours of v'}
If there is a $v' \in S'$ with $d^-(v') = d+\epsilon $, then every $\Omega $-set is contained in $N^-(v')$.
\end{lemma}
\begin{proof}
Let $u \in V(G)$ with $N^+(u) = \{ u_1,u_2,\dots ,u_d\} $.  Suppose that $u \not \in N^-(v')$.  In each of the $d$ sets $T(u_i)$ there lies at most one in-neighbour of $v'$.  It follows that every outlier of $u$ must be an in-neighbour of $v'$.  The case $u \in N^-(v')$ is similar.
\end{proof}

\section{Out-degree $d = 2$, excess $\epsilon = 2$}

For the remainder of this paper, we will assume that $G$ is a $k$-geodetic digraph with minimum out-degree $d = 2$ and excess $\epsilon = 2$, where $k \geq 2$.  We will occasionally have to consider the case $k = 2$ separately.  We now state our

\begin{theorem}[Main Theorem]
There are no non-diregular $(2,k,+2)$-digraphs for $k \geq 2$.
\end{theorem}

We will proceed to derive a list of possible in-degree sequences for $G$.  Analysing each in turn, we will obtain a contradiction in each case, thereby proving the main theorem.  Before embarking upon this program, we mention a final important lemma that connects the case of excess two with previous work on excess one.  This result generalises the proof strategy of Theorem 2 of \cite{MirSil}.

\begin{lemma}[Amalgamation Lemma]\label{amalgamation lemma}
Suppose that $G$ contains vertices $u_1,u_2$ such that for all vertices $u \in V(G)$ we have $O(u) \cap \{ u_1,u_2\} \not = \varnothing $.  Then $N^+(u_1) \not = N^+(u_2)$.
\end{lemma}
\begin{proof}
Suppose that $N^+(u_1) = N^+(u_2)$.  Denote the graph resulting from the amalgamation of vertices $u_1,u_2$ by $G^*$.  Inspection shows that if $G^*$ is not $k$-geodetic, neither is $G$.  $G^*$ is therefore a $(2,k,+1)$-digraph, contradicting the results of  \cite{MirSil} and \cite{Sil}. 

\end{proof}

\section{There are no vertices in $G$ with in-degree four}

By Lemma \ref{limit on out-degree}, all vertices in $S'$ have in-degree three or four.  In this section we shall prove that all vertices in $S'$ must have in-degree three.  If $G$ contained a vertex with in-degree zero, deleting this vertex would yield a digraph with out-degree two and excess one, which is impossible \cite{MirSil,Sil}; hence every vertex in $S$ has in-degree one, so that by Lemma \ref{average in-degree} we have $|S| = \sum _{v' \in S'}(d^-(v')-2)$.  By Corollary \ref{bound on |S|, |S'|} we have $|S| \leq 4$, so it follows that if $G$ contains a vertex of in-degree four, then the possible in-degree sequences of $G$ are $(1,1,2,\dots ,2,4)$, $(1,1,1,1,2,\dots ,2,4,4)$, $(1,1,1,2,\dots ,2,3,4)$ and $(1,1,1,1,2,\dots ,2,3,3,4)$.  We can narrow down the possibilities further as follows.

\begin{lemma}\label{|S| = 4}
If $G$ contains a vertex $v'$ with in-degree four, then $|S| = 4$.
\end{lemma}
\begin{proof}
Suppose that $|S| \leq 3$ and let $d^-(v') = 4$.  By $k$-geodecity, every vertex has at most one $\leq k$-path to $v'$.  The smallest possible number of initial vertices of $\leq k$-paths to $v'$ is achieved if $S \subset N^-(v')$ and $d(v'',v')\geq k$ for $v'' \in S'-\{ v'\} $, so that
\[M(2,k)+2 \geq |T_{-k}(v')| \geq 4+3M(2,k-2)+M(2,k-1) = 2 + M(2,k) + M(2,k-2),\]
which is impossible for $k \geq 2$.
\end{proof}

The only possible in-degree sequences for $G$ are thus $(1,1,1,1,2,\dots ,2,4,4)$ and 
\newline $(1,1,1,1,2,\dots ,2,3,3,4)$.  We need one final piece of structural information and then we can proceed to analyse the possible in-degree sequences.

\begin{corollary}\label{in-degree 4 in-neighbours}
If $|S| = 4$ and there is a vertex $v' \in S'$ with in-degree four, then $S = N^-(v')$ and all $\Omega $-sets are contained in $S$.  If $\Omega \subset S$ is an outlier set, then so is $S - \Omega $.
\end{corollary}
\begin{proof}
Putting $\epsilon = 2$ in Lemma \ref{outliers are in-neighbours of v'}, we see that $O(u) \subseteq N^-(v')$ for all $u \in V(G)$.  Hence for any vertex $u$ we have by Lemma \ref{every v an outlier} 

\[ S \subseteq O(N^+(u)) \subseteq N^-(v') .\]
As $|S| = |N^-(v')| = 4$, we must have equality in the above inclusion, i.e. $S = N^-(v')$.

Let $O(u) = \Omega $.  Write $u^- \in N^-(u)$ and $N^+(u^-) = \{ u,u^+\} $.  By Lemma \ref{every v an outlier} we have 

\[ \Omega \cup O(u^+) = O(u) \cup O(u^+) = O(N^+(u^-)) = S, \]
so we must have $O(u^+) = S - \Omega $. 

\end{proof}

We are now in a position to show that neither of the remaining in-degree sequences can arise.

\begin{theorem}
There are no $(2,k,+2)$-digraphs with in-degree sequence $(1,1,1,1,2,\dots ,2,4,4)$ for $k \geq 2$.
\end{theorem}
\begin{proof}
Let $v_1',v_2'$ be the vertices with in-degree four.  By Corollary \ref{in-degree 4 in-neighbours}, $S = N^-(v_1') = N^-(v_2')$ and $v_1'$ is not an outlier, so it follows that $v_2' \in T^-(v)$ for some $v \in S$.  But as $N^-(v_2') = S$, it follows that there is a $\leq k$-cycle through $v$, contradicting $k$-geodecity.
\end{proof}

\begin{theorem}
There are no $(2,k,+2)$-digraphs with in-degree sequence $(1,1,1,1,2,\dots ,2,3,3,4)$ for $k \geq 2$.
\end{theorem}
\begin{proof}
Let $v'$ be the vertex with in-degree four and let $w_1, w_2$ be the vertices with in-degree three.  Write $S = \{ v_1,v_2,v_3,v_4\} $.  By Corollary \ref{in-degree 4 in-neighbours}, $N^-(v') = S$ and no vertex outside $S$ is an outlier.  Without loss of generality, suppose that $O(v') = \{ v_1,v_2\} $.  By Corollary \ref{in-degree 4 in-neighbours}, $\{ v_3,v_4 \} $ is also an $\Omega $-set.  By Lemma \ref{every v' neighbour of an outlier} we can thus assume that  

\[ v_1,v_3 \in N^-(w_1), v_2, v_4 \in N^-(w_2) .\]
Again by Lemma \ref{every v' neighbour of an outlier}, $\{ v_1,v_3\} $ and $\{ v_2,v_4\} $ cannot be $\Omega $-sets.  The only other possible $\Omega $-sets are $\{ v_1,v_4\} $ and $\{ v_2,v_3\} $.  We see then that $\Omega \cap \{ v_1,v_3\} \not = \varnothing $ for all $\Omega $-sets and $N^+(v_1) = N^+(v_3)$, contradicting the Amalgamation Lemma.

\end{proof}

It follows that no vertex of $G$ has in-degree $\geq 4$.  By Lemma \ref{average in-degree} and Corollary \ref{bound on |S|, |S'|}, we must therefore have $|S| = |S'|$ and $|S| \leq 4$, which leaves us with only four in-degree sequences to analyse, namely $(1,2,\dots ,2,3), (1,1,2,\dots ,2,3,3), (1,1,1,2,\dots ,2,3,3,3)$ and $(1,1,1,1,2,\dots ,2,3,3,3,3)$.  For $|S| = r$, we will write $S = \{ v_1,\dots ,v_r\} , S' = \{ v_1',\dots,v_r'\} $.

\section{Degree sequence $(1,2,\dots ,2,3)$}

\begin{theorem}
There are no $(2,k,+2)$-digraphs with in-degree sequence $(1,2,\dots ,2,3)$ for $k \geq 3$.
\end{theorem}
\begin{proof}
We obtain a lower bound for $|T_{-k}(v_1')|$ by assuming that $v_1 \in N^-(v_1')$.  By $k$-geodecity, all vertices in $T_{-k}(v_1')$ are distinct, so 
\[ M(2,k)+2 \geq |T_{-k}(v_1')| \geq 2 + M(2,k-2) + 2M(2,k-1) = 1 + M(2,k) + M(2,k-2).\]
This inequality is not satisfied for $k \geq 3$.
\end{proof}

This leaves open the question of whether there exists a non-diregular $(2,2,+2)$-digraph with the given in-degree sequence.  By the argument of the preceding theorem, such a digraph must contain the subdigraph shown in Figure \ref{fig:|S| = 1, k = 2}, which also displays the vertex-labelling that we shall employ.  We proceed to show that no such digraph exists.

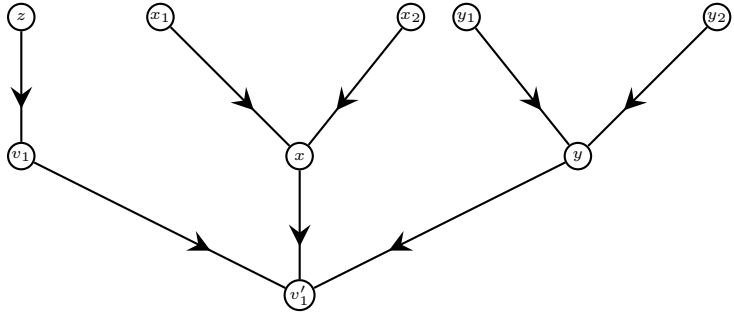
\begin{figure}\centering
\begin{tikzpicture}[middlearrow=stealth,x=0.2mm,y=-0.2mm,inner sep=0.1mm,scale=1.85,
	thick,vertex/.style={circle,draw,minimum size=10,font=\tiny,fill=white},edge label/.style={fill=white}]
	\tiny
	\node at (200,0) [vertex] (v0) {$v_1'$};
	\node at (100,-50) [vertex] (v1) {$v_1$};
	\node at (200,-50) [vertex] (v2) {$x$};
	\node at (300,-50) [vertex] (v3) {$y$};
	\node at (100,-100) [vertex] (v4) {$z$};
           \node at (150,-100) [vertex] (v5) {$x_1$};
	\node at (240,-100) [vertex] (v6) {$x_2$};
	\node at (260,-100) [vertex] (v7) {$y_1$};
	\node at (350,-100) [vertex] (v8) {$y_2$};

	\path
		(v1) edge [middlearrow] (v0)
		(v2) edge [middlearrow] (v0)
		(v3) edge [middlearrow] (v0)
                     (v4) edge [middlearrow] (v1)
                     (v5) edge [middlearrow] (v2)
                     (v6) edge [middlearrow] (v2)
                     (v7) edge [middlearrow] (v3)
                     (v8) edge [middlearrow] (v3)

		;
\end{tikzpicture}
\caption{Subdigraph of any $(2,2,+2)$-digraph with in-degree sequence $(1,2,\dots ,2,3)$}
\label{fig:|S| = 1, k = 2}
\end{figure}

Evidently $v_1'$ is not an outlier.  Note that all arcs added to the subdigraph in Figure \ref{fig:|S| = 1, k = 2} must terminate in the set $\{ z,x_1,x_2,y_1,y_2\} $.  $G$ is out-regular with degree $d = 2$, so we can assume without loss of generality that $z \rightarrow x_1$.  By $2$-geodecity, $x_1 \not \rightarrow z$ and  $x_1 \not \rightarrow x_2$, so we can assume that $x_1 \rightarrow y_1$.  Similarly, we must either have $y_1 \rightarrow z$ or $y_1 \rightarrow x_2$.

\begin{lemma}\label{y1 not adjacent to z}
The out-neighbourhood of $y_1$ is $N^+(y_1) = \{ y,x_2\} $.
\end{lemma}
\begin{proof}
Assume for a contradiction that $y_1 \rightarrow z$.  $x \not \rightarrow x_1$ or $x_2$ by $2$-geodecity.  Also, $x \not \rightarrow z$, or we would have two paths $x_1 \rightarrow y_1 \rightarrow z$ and $x_1 \rightarrow x \rightarrow z$.  Similarly, $x \not \rightarrow y_1$, or there would be paths $x_1 \rightarrow y_1$ and $x_1 \rightarrow x \rightarrow y_1$.  Therefore $x \rightarrow y_2$.  We now analyse the possible out-neighbours of $y$.  $y \not \rightarrow y_1,y_2$ and if $y\rightarrow x_1$, then there would be paths $y_1 \rightarrow z \rightarrow x_1$ and $y_1 \rightarrow y \rightarrow x_1$.  Likewise $y \not \rightarrow z$, so $y \rightarrow x_2$.  We now see that $v_1' \not \rightarrow x_2$ or $y_2$; for example, if $v_1' \rightarrow y_2$, then there would be paths $x \rightarrow y_2$ and $x \rightarrow v_1' \rightarrow y_2$.  Since $v_1'$ cannot be adjacent to two vertices linked by an arc, we see that $v_1'$ cannot have two out-neighbours in $N^{-2}(v_1')$ without violating $2$-geodecity.  Hence we are forced to conclude that $y_1 \rightarrow x_2$.

\end{proof}

\begin{theorem}
There are no $(2,2,+2)$-digraphs with in-degree sequence $(1,2,\dots ,2,3)$.
\end{theorem}
\begin{proof}
By Lemma \ref{y1 not adjacent to z}, we have $y_1 \rightarrow x_2$.  There are five possibilities for $N^+(v_1')$, namely $\{ z,x_2\} , \{ z,y_1\} ,$ $\{ z,y_2\} , \{ x_1,y_2\} $ and $\{ x_2,y_2\} $; we discuss each case in turn.

Case i): $N^+(v_1') = \{ z, x_2\} $

If $v_1 \rightarrow y_1$, then we have paths $z \rightarrow x_1 \rightarrow y_1$ and $z \rightarrow v_1 \rightarrow y_1$, so $v_1 \not \rightarrow y_1$.  Likewise, $v_1$ is not adjacent to $z,x_1$ or $x_2$.  Thus $v_1 \rightarrow y_2$.  Similarly, $x_2 \rightarrow y_2$.  We must now have $x \rightarrow y_1$.  However, this gives us paths $x_1 \rightarrow y_1$ and $x _1 \rightarrow x \rightarrow y_1$, which is impossible.

Case ii) $N^+(v_1') = \{ z,y_1\} $

By $2$-geodecity, $y \rightarrow x_1$; however, this yields paths $y \rightarrow v_1' \rightarrow y_1$ and $y \rightarrow x_1 \rightarrow y_1$.

Case iii): $N^+(v_1') = \{ z, y_2\} $

As there are paths $x \rightarrow v_1' \rightarrow z$, $x \rightarrow v_1' \rightarrow y_2$ we cannot have $x \rightarrow z$ or $x \rightarrow y_2$.  Obviously $x \not \rightarrow x_1, x_2$, so $x \rightarrow y_1$.  Now there are paths $x_1 \rightarrow y_1$ and $x_1 \rightarrow x \rightarrow y_1$, a contradiction.

Case iv): $N^+(v_1') = \{ x_1,y_2\} $

By $2$-geodecity, we have successively $v_1 \rightarrow x_2$, $x \rightarrow z$ and $y \rightarrow z$.  But now as each of $z,x_1$ and $x_2$ already has in-degree two, we are led to conclude that $y_2 \rightarrow y_1$, violating $2$-geodecity.

Case v): $N^+(v_1') = \{ x_2,y_2\}  $

By $2$-geodecity, $v_1$ cannot be adjacent to any of $z, x_1, x_2, y_1$ or $y_2$. 

Having exhausted all possibilities, our proof is complete.

\end{proof}

\section{Degree sequence $(1,1,2,\dots ,2,3,3)$}

We shall assume firstly that $k \geq 3$ and deal with the special case of $k = 2$ separately.

\begin{lemma}\label{|S| =2 Lemma}
If $k \geq 3$, then for each $v' \in S'$ we have $S \subset N^-(v')$.
\end{lemma}
\begin{proof}
Let $v' \in S'$ and consider $T_{-k}(v')$.  Suppose that neither $v_1$ nor $v_2$ lies in $N^-(v')$.  Then for $k \geq 2$, by $k$-geodecity

\[ M(2,k)+2 \geq 4 + 2M(2,k-3) + 2M(2,k-1) = 2 + M(2,k) + M(2,k-2),\]
a contradiction.  Now suppose that $|S \cap N^-(v')| = 1$.  We would then have

\[ M(2,k) + 2 \geq 3+2M(2,k-1) + M(2,k-3) = 2 + M(2,k) + M(2,k-3), \]
which again is impossible for $k \geq 3$.

\end{proof}

Hence we can set $N^-(v_1') = \{ v_1,v_2,x\} , N^-(v_2') = \{ v_1,v_2,y\} $.  This situation is displayed in Figure \ref{fig:in-degreesequence1133}, where $N^-(v_i) = \{ v_i^-\} $ for $i = 1,2$.

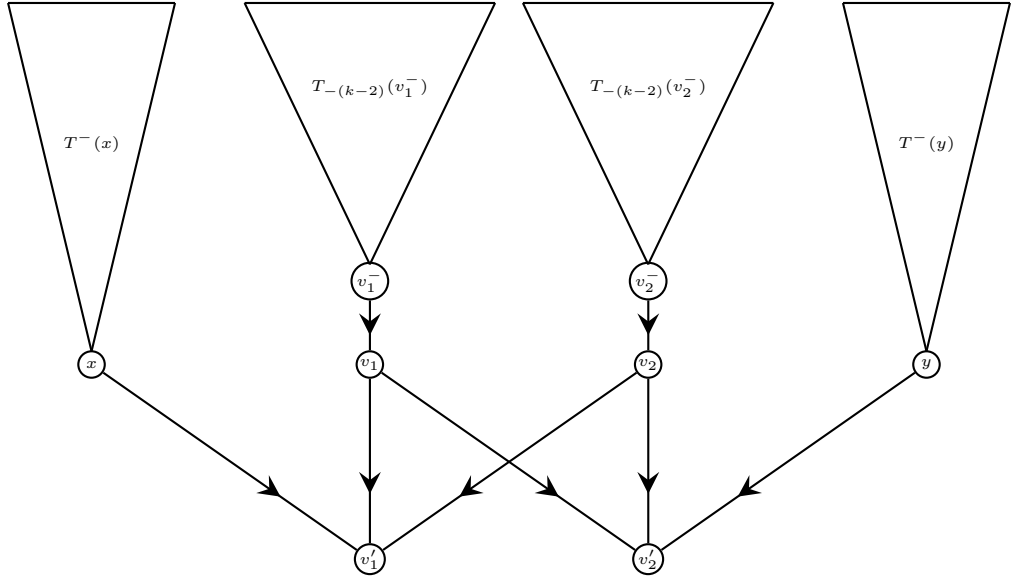
\begin{figure}\centering
\begin{tikzpicture}[middlearrow=stealth,x=0.2mm,y=-0.2mm,inner sep=0.1mm,scale=1.85,
	thick,vertex/.style={circle,draw,minimum size=10,font=\tiny,fill=white},edge label/.style={fill=white}]
	\tiny
	\node at (150,0) [vertex] (v0) {$v_1'$};
	\node at (250,0) [vertex] (v1) {$v_2'$};
	\node at (50,-70) [vertex] (v2) {$x$};
	\node at (150,-70) [vertex] (v3) {$v_1$};
	\node at (250,-70) [vertex] (v4) {$v_2$};
	\node at (350,-70) [vertex] (v5) {$y$};
	\node at (150,-100) [vertex] (v6) {$v_1^-$};
	\node at (250,-100) [vertex] (v7) {$v_2^-$};
           \node at (50,-150) {$T^-(x)$};
           \node at (150,-170) {$T_{-(k-2)}(v_1^-)$};
           \node at (250,-170) {$T_{-(k-2)}(v_2^-)$};
           \node at (350,-150) {$T^-(y)$};

	\path
		(v2) edge [middlearrow] (v0)
		(v3) edge [middlearrow] (v0)
		(v4) edge [middlearrow] (v0)
		(v3) edge [middlearrow] (v1)
		(v4) edge [middlearrow] (v1)
		(v5) edge [middlearrow] (v1)
		(v6) edge [middlearrow] (v3)
		(v7) edge [middlearrow] (v4)
;

           \draw (50,-74.75)--(20,-200);
           \draw (50,-74.75)--(80,-200);

           \draw (150,-106)--(105,-200);
           \draw (150,-106)--(195,-200);

           \draw (250,-106)--(205,-200);
           \draw (250,-106)--(295,-200);

           \draw (350,-74.75)--(320,-200);
           \draw (350,-74.75)--(380,-200);

           \draw (20,-200)--(80,-200);
           \draw (105,-200)--(195,-200);
           \draw (205,-200)--(295,-200);
           \draw (320,-200)--(380,-200);

\end{tikzpicture}
\caption{Configuration for in-degree sequence $(1,1,2,\dots ,2,3,3)$}
\label{fig:in-degreesequence1133}
\end{figure}

\begin{corollary}\label{|S| = 2 distance between S'}
$d(v_1',v_2') \geq k$ and $d(v_2',v_1') \geq k$.  If $d(v_1',v_2') = k$, then $v_1' \in N^{-(k-1)}(y)$, and similarly if $d(v_2',v_1') = k$, then $v_2' \in N^{-(k-1)}(x)$.  
\end{corollary}

\begin{corollary}\label{vertices which have v_1 as outlier}
$|O^-(v_1)| = |O^-(v_2)| = 2^k+1$.
\end{corollary}
\begin{proof}
By $k$-geodecity, $v_2,v_1',v_2' \not \in T^-(v_1)$, so $|T_{-k}(v_1)| = 1 + M(2,k-1)$, yielding $|O^-(v_1)| = M(2,k) + 2 - (1 + M(2,k-1)) = 2^k + 1$.  Similarly for $v_2$.
\end{proof}

\begin{corollary}\label{x and y are outliers of 1 or 2 vertices}
If $d(v_1',v_2') = k$, then $|O^-(y)| = 1$ and if $v_2' \in O(v_1')$, then $|O^-(y)| = 2$.  Similarly, $|O^-(x)| = 1$ if $d(v_2',v_1') = k$ and $|O^-(x)| = 2$ if $v_1' \in O(v_2')$.
\end{corollary}
\begin{proof}
Similar to the proof of Corollary \ref{vertices which have v_1 as outlier}.
\end{proof}

\begin{lemma}\label{O-(v1') and O-(v2') are singletons}
$|O^-(v_1')| = |O^-(v_2')| = 1$.
\end{lemma}
\begin{proof}
Consider $|T_{-k}(v')|$, where $v' \in S'$.  Counting distinct vertices of $G$,  
\[ M(2,k)+2 = 3+2M(2,k-2)+M(2,k-1)+ |O^-(v')| = 1+ M(2,k) + |O^-(v')|. \] 
\end{proof}

\begin{lemma}
The vertices $x$ and $y$ are distinct.
\end{lemma}
\begin{proof}
Suppose that $x = y$.  Then $N^-(v_1') = N^-(v_2')$, so that we must have $v_1' \in O(v_2'), v_2' \in O(v_1')$ and hence by Corollary \ref{x and y are outliers of 1 or 2 vertices} $|O^-(x)| = 2$.  As $N^+(v_1) = N^+(v_2) = N^+(x)$, by $k$-geodecity we have $O(v_1) = \{ v_2,x\} , O(v_2) = \{ v_1,x\} , O(x) = \{ v_1,v_2\} $, so $O^-(x) = \{ v_1,v_2\} $.  By Lemma \ref{every v' neighbour of an outlier}, every $\Omega $-set must intersect $\{ v_1,v_2,x\} $, so it follows that every $\Omega $-set contains an element of $\{ v_1,v_2\} $, contradicting the Amalgamation Lemma.  
\end{proof}

\begin{lemma}\label{possible outlier sets |S| = 2}
Let $\Omega $ be an outlier set.  Then either $\Omega \cap S \not = \varnothing $ or $\Omega = \{ x,y\} $.  $\{ x,y\} $ is an $\Omega $-set.  
\end{lemma}
\begin{proof}
By Lemma \ref{every v' neighbour of an outlier} and the Amalgamation Lemma.  
\end{proof}

Let $\alpha $ denote the number of vertices of $G$ with outlier set $\{ v_1,v_2\} $ and $\beta $ the number of vertices with outlier set $\{ x,y\} $.

\begin{lemma}\label{alpha = beta +1}
$\alpha = \beta + 1$.
\end{lemma}
\begin{proof}
By Corollary \ref{vertices which have v_1 as outlier}, $v_1$ and $v_2$ appear in $2(2^k+1) - \alpha = 2^{k+1} + (2 - \alpha )$ $\Omega $-sets.  By Lemma \ref{possible outlier sets |S| = 2}, any $\Omega $-set that does not contain either $v_1$ or $v_2$ must equal $\{ x,y\} $.  It follows that

\[ M(2,k)+2 = 2^{k+1}+1 =  2^{k+1} + 2 - \alpha  + \beta, \] 
implying the result.
\end{proof}

\begin{corollary}\label{v1' and v2' not outliers of each other}
$v_2 \in O(v_1)$, $v_1 \in O(v_2)$ and $d(v_1',v_2') = d(v_2',v_1') = k$.
\end{corollary}
\begin{proof}

Suppose that $d(v_1,v_2) \leq k$.  Then we must have 
\[ S' \cap T^-(v_2) = N^+(v_1) \cap T^-(v_2) \not = \varnothing ,\] 
contradicting $k$-geodecity.  Thus $v_2 \in O(v_1)$ and similarly $v_1 \in O(v_2)$.

Suppose that $v_2' \in O(v_1')$.  Then $v_1$ and $v_2$ have no out-neighbours in $T^-(y)$, so

\[ O(v_1) = \{ v_2,y\} , O(v_2) = \{ v_1,y\} .\]
By Corollary \ref{x and y are outliers of 1 or 2 vertices}, $|O^-(y)| = 2$, so $\{ x,y\} $ is not an $\Omega $-set, contradicting Lemma \ref{possible outlier sets |S| = 2}.  $v_1' \in O(v_2')$ is impossible for the same reason.
\end{proof}

\begin{theorem}
There are no $(2,k,+2)$-digraphs with in-degree sequence $(1,1,2,\dots ,2,3,3)$ for $k \geq 3$.
\end{theorem}
\begin{proof}
It follows from Corollaries \ref{x and y are outliers of 1 or 2 vertices} and \ref{v1' and v2' not outliers of each other} and Lemma \ref{possible outlier sets |S| = 2} that there is a unique vertex $z$ such that $O(z) = \{ x,y\} $.  Furthermore, no other $\Omega $-set contains $x$ or $y$.  Hence, by Lemma \ref{alpha = beta +1}, $\alpha = 2, \beta = 1$.  Denote the two vertices with $\Omega $-set $\{ v_1,v_2\} $ by $w,w'$.  Write $N^+(w) = \{ w_1,w_2\} , N^+(w') = \{ w_1',w_2'\} $.  

It is easily seen that $\{ w,w'\} \cap \{ x,y\} = \varnothing $.  Suppose that $w = x$ and set $w_2 = v_1'$.  By Corollary \ref{v1' and v2' not outliers of each other}, we must have $y \in N^{k-1}(v_1')$, so by $k$-geodecity $x,y,v_1,v_2 \not \in T(w_1)$, so $O(w_1) = \{ v_1',v_2' \} $, contradicting Lemma \ref{possible outlier sets |S| = 2}.  The other cases are identical. 

As $O(w) = \{ v_1,v_2\} $, $x,y \in T(w_1) \cup T(w_2)$.  Suppose that $x$ and $y$ lie in the same branch, e.g. $x,y \in T(w_1)$.  By $k$-geodecity and the definition of $w$, $\{ x,y,v_1,v_2 \} \cap (\{ w\} \cup T(w_2)) = \varnothing $, so that $O(w_2) = \{ v_1',v_2'\} $, which is impossible by Lemma \ref{possible outlier sets |S| = 2}.  Hence we can assume $x \in T(w_1), y \in T(w_2)$.  Then $N^-(v_2') \cap T(w_1) = N^-(v_1') \cap T(w_2) = \varnothing $, so $v_2' \in O(w_1), v_1' \in O(w_2)$.  Applying the same analysis to $w'$, we see that we can assume $v_2' \in O(w_1'), v_1' \in O(w_2')$.  By Lemma \ref{O-(v1') and O-(v2') are singletons}, it follows that $w_1 = w_1'$ and $w_2 = w_2'$, so that $N^+(w) = N^+(w')$.  As $O(w) = \{ v_1,v_2\} $, we must have $w' \in T_k(w)$.  Hence there is a $\leq k$-cycle through either $w_1$ or $w_2$.
\end{proof}

Now we turn to the case $k = 2$.  The argument of Lemma \ref{|S| =2 Lemma} shows that each member of $S'$ has an in-neighbour in $S$.  This allows us to deduce the following lemma.

\begin{lemma}\label{members of S' independent}
Neither element of $S'$ is adjacent to the other.
\end{lemma}
\begin{proof}
Suppose that $v_2' \rightarrow v_1'$.  If $|N^-(v_1') \cap S| = 1$, then the order of $G$ would be at least 10, whereas $|V(G)| = M(2,2) + 2 = 9$.  Hence $|N^-(v_1') \cap S| = 2$ and since $v_2'$ also has an in-neighbour in $S$, there would be an element of $S$ with two $\leq 2$-paths to $v_1'$.
\end{proof}

\begin{theorem}
There are no $(2,2,+2)$-digraphs with in-degree sequence $(1,1,2,\dots ,2,3,3)$.
\end{theorem}
\begin{proof}

If $S \subset N^-(v_1') \cap N^-(v_2')$, then the argument for $k \geq 3$ remains valid, so we can assume that $N^-(v_1') = \{ v_1,x,y\} $, where $\{ x,y\} \cap (S \cup S') = \varnothing $.  Simple counting shows that $O^-(v_1') =\varnothing $.  We will write $N^-(x) = \{ x_1,x_2\} , N^-(y) = \{ y_1,y_2\} , N^-(v_1) = \{ z\} $.  Without loss of generality, there are four possibilities: i) $v_2 = z, v_2' = x_1$, ii) $v_2 =x_1, v_2' = z$, iii) $v_2 = y_1, v_2' = x_1$ and iv) $v_2 = x_1, v_2' = x_2$.

Case i) $v_2 = z, v_2' = x_1$:

$v_2'$ has three in-neighbours.  By Lemma \ref{members of S' independent}, $v_1' \not \in N^-(v_2')$.  By $2$-geodecity, $N^-(v_2') \cap T^-(x) = \varnothing $.  $v_1$ and $z = v_2$ cannot both be in-neighbours of $v_2'$, so $v_2'$ must have exactly two in-neighbours in $T^-(y)$; necessarily $y_1, y_2 \in N^-(v_2')$ but $y \not \in N^-(v_2')$.  If $v_2 \rightarrow v_2'$, then there is no vertex other than $x$ that $v_2'$ can be adjacent to without violating $2$-geodecity, so we must have $v_1 \rightarrow v_2'$ and $v_2' \rightarrow v_2$.  As we already have a 2-path $v_2 \rightarrow v_1 \rightarrow v_2'$, $v_2$ cannot be adjacent to $v_2', y_1$ or $y_2$, so $v_2 \rightarrow x_2$.  As all in-neighbours of $v_2$ and $v_2'$ are accounted for, we must have $y \rightarrow x_2$.  But now the only possible out-neighbourhood of $v_1'$ is $\{ y_1,y_2\} $, which gives two $2$-paths from $v_1'$ to $v_2'$.

Case ii) $v_2 =x_1, v_2' = z$:

As $v_1 \not \rightarrow v_2'$, Lemma \ref{|S| =2 Lemma} shows that $v_2 \rightarrow v_2'$.  Without loss of generality, $N^-(v_2')$ must be one of $\{ v_2, x_2,y_1\} , \{ v_2, x_2, y\} $ or $\{ v_2, y_1,y_2 \} $.  Suppose that $N^-(v_2') = \{ v_2, x_2,y_1\} $.  Then $v_2' \rightarrow y_2$ and $N^+(v_1')$ is either $\{ v_2,y_2\} $ or $\{ x_2,y_2\} $.  If $N^+(v_1') = \{ v_2,y_2\} $, then we can deduce that $x \rightarrow y_1$, $y \rightarrow x_2$ and $y_2 \rightarrow x_2$, so that there are paths $y_2 \rightarrow x_2$ and $y_2 \rightarrow y \rightarrow x_2$, so assume that  $N^+(v_1') = \{ x_2,y_2\} $.  As $y$ can already reach $x_2$ by a $2$-path, there is an arc $y \rightarrow v_2$.  $v_2$ has a unique in-neighbour, so $y_2 \rightarrow x_2$ and hence there are paths $v_1' \rightarrow x_2$ and $v_1' \rightarrow y_2 \rightarrow x_2$.

If $N^-(v_2') = \{ v_2, x_2, y\} $, then $y_1$ cannot be adjacent to any of $v_2', v_2, x_2$ or $y_2$ without violating $2$-geodecity.  Hence we can assume that $N^-(v_2') = \{ v_2, y_1,y_2 \} $.  Now we must have $v_2' \rightarrow x_2$.  Without loss of generality, $x_2 \rightarrow y_1$ and $x \rightarrow y_2$.  We cannot have $y \rightarrow x_2$, or there would be paths $y_1 \rightarrow y \rightarrow x_2$ and $y_1 \rightarrow v_2' \rightarrow x_2$, so $y \rightarrow v_2$.  $v_1'$ cannot be adjacent to both $y_1$ and $y_2$, so $v_1' \rightarrow x_2$ and hence also $v_1' \rightarrow y_2$.  Now the only possible remaining arc is $v_1 \rightarrow y_1$, so that we have paths $v_2' \rightarrow v_1 \rightarrow y_1$ and $v_2' \rightarrow x_2 \rightarrow y_1$, which is impossible.

Case iii) $v_2 = y_1, v_2' = x_1$:

$N^-(v_2')$ must be either $\{ z,v_2,y_2\} $ or $\{ v_1,v_2,y_2\} $.  In the first case, there are no vertices other than $x$ that $v_2'$ can be adjacent to without violating $2$-geodecity, so $N^-(v_2') = \{ v_1,v_2,y_2\} $.  By $2$-geodecity, $v_2' \rightarrow z$.  If $y \rightarrow z$, then there would be distinct $\leq 2$-paths from $v_2$ to $z$, so $y \rightarrow x_2$.  $v_1'$ is not adjacent to both $v_2$ and $y_2$ and is not adjacent to $x_2$, or there would be two $\leq 2$-paths from $y$ to $x_2$, so we see that $v_1' \rightarrow z$, implying that there are paths $v_1 \rightarrow v_2' \rightarrow z$ and $v_1 \rightarrow v_1' \rightarrow z$.

Case iv) $v_2 = x_1, v_2' = x_2$:

As $v_2 \not \rightarrow v_2'$, we have $v_1 \rightarrow v_2'$ and $N^-(v_2') = \{ v_1,y_1,y_2\} $.  Hence $v_2' \rightarrow z$.  As $y_1$ can already reach $z$ by a $2$-path, $y \not \rightarrow z$, so $y \rightarrow v_2$.  $z$ must be adjacent to $y_1$ or $y_2$, but can already reach $v_2'$ via $v_1$, yielding a contradiction.

Having dealt with every possibility, the result is proven.
\end{proof}

\section{Degree sequence $(1,1,1,2,\dots ,2,3,3,3)$}

This represents the most difficult case to deal with.  Again, we will discuss the cases $k = 2$ and $k \geq 3$ separately.

\begin{lemma}\label{forms of outlier sets |S| = 3}
If $k \geq 2$, then for every $u \in V(G)$ we have $|O(u) \cap S| = 1$ or $2$.  There exists an $\Omega $-set contained in $S$.
\end{lemma}
\begin{proof}
Let $u \in V(G)$ be arbitrary.  Let $u^-$ be an in-neighbour of $u$ and let $u^+$ be the other out-neighbour of $u^-$.  By Lemma \ref{every v an outlier}, if $S \cap O(u) = \varnothing $, then we would have $S \subseteq O(u^+)$.  Since $|S| = 3$ and $|O(u^+)| = 2$, this is impossible.  The other half of the lemma follows trivially.
\end{proof}

\begin{lemma}\label{in-neighbours in S}
If $k \geq 2$, then for each $v' \in S'$, $S \cap N^-(v') \not = \varnothing $.
\end{lemma}
\begin{proof}
Assume that $v'$ is an element of $S'$ such that $S \cap N^-(v') = \varnothing $.  Then we obtain a lower bound for $|T_{-k}(v')|$ by assuming that all members of $S$ lie in $N^{-2}(v')$, whilst $v'$ is at distance $\geq k$ from the remaining members of $S'$.  Recalling that $M(2,k) = 0 $ for $k < 0$, this yields

\[ |T_{-k}(v')| \geq 6 + 3M(2,k-3) + M(2,k-2) + M(2,k-1) = 3 + M(2,k) + M(2,k-3) ,\]
a contradiction.
\end{proof}

For $i = 1,2,3$, we will say that a vertex $v' \in S'$ is Type $i$ if $|S \cap N^-(v')| = i$.  As each member of $S$ has out-degree two, it follows that if for $i = 1,2,3$ there are $N_i$ vertices of Type $i$ then $N_1 + 2N_2 + 3N_3 \leq 6$.  We now determine the number of vertices of each type.

\begin{lemma}\label{properties of Type 1 vertices}
Let $k \geq 2$.  Suppose that $v' \in S'$ is Type 1, with $N^-(v') \cap S = \{ v\} $.  Then for $v^* \in S - \{ v\} $ we have $d(v^*,v') = 2$ and for $v'' \in S' - \{ v'\} $ we have $d(v'',v') = k$.  Also $O^-(v') = \varnothing $.
\end{lemma}
\begin{proof}
The results for $k = 2$ follow by simple counting, so assume that $k \geq 3$.
Let $v', v$ be as described.  Consider $T_{-k}(v')$.  We obtain a lower bound for $|T_{-k}(v')|$ by assuming that $S - \{ v\} \subset N^{-2}(v')$  and that $(S' - \{ v'\} )\cap T^-(v') = \varnothing $.  Hence 
\[ |V(G)| \geq |T_{-k}(v')| \geq 5 + 2M(2,k-3) + M(2,k-2) + M(2,k-1) = 2 + M(2,k) = |V(G)|. \]
Clearly, if $v'$ were any closer to the remaining members of $S'$ or if $v'$ were any further from the vertices in $S - \{ v\} $, $|T_{-k}(v')|$ would have order greater than $2 + M(2,k)$, which is impossible by $k$-geodecity.  Evidently all vertices of $G$ lie in $T_{-k}(v')$, so $O^-(v') = \varnothing $.  

\end{proof}

Our reasoning for the cases $k \geq 3$ and $k = 2$ must now part company, so we will now assume that $k \geq 3$ and return to the case $k = 2$ presently.

\begin{lemma}\label{elements of S' not adjacent}
For $k \geq 3$, no two elements of $S'$ are adjacent to one another.
\end{lemma}
\begin{proof}
Suppose that there is an arc $(v',v'')$ in $G$, where $v',v'' \in S'$.  Consider $T_{-k}(v'')$.  We obtain a lower bound for $|T_{-k}(v'')|$ by assuming that $v''$ is Type 2 and that $v'$ is Type 1, whilst $v''$ lies at distance $\geq k$ from the remaining vertex in $S'$.  Then by inspection 
\[ |T_{-k}(v'')| \geq 4 + M(2,k-1) + 2M(2,k-2) + M(2,k-3) = 2 + M(2,k) + M(2,k-3) ,\]
which is impossible for $k \geq 3$.
\end{proof}

\begin{lemma}
There are no Type 3 vertices.
\end{lemma}
\begin{proof}
Suppose for a contradiction that $v_1' \in S'$ is a Type 3 vertex, i.e. $N^-(v_1') = S$.  As $N_1 + 2N_2 + 3N_3 \leq 6$, $S'$ must contain a Type 1 vertex.  We can assume that $v_2'$ is Type 1 and $N^+(v_1) = \{ v_1',v_2'\} , N^+(v_2) = \{ v_1',v_3'\} $.  It follows by Lemma \ref{properties of Type 1 vertices} that $v_2 \in N^{-2}(v_2')$.  Therefore we must have $N^+(v_2) \cap N^-(v_2') = \{ v_1',v_3'\} \cap N^-(v_2') \not = \varnothing $, which contradicts Lemma \ref{elements of S' not adjacent}.   
\end{proof}

\begin{lemma}
There is a Type 2 vertex.
\end{lemma}
\begin{proof}
Assume for a contradiction that each vertex in $S'$ is Type 1.  Suppose that the sets $S \cap N^-(v_i')$, $i = 1,2,3$ are not all distinct; say $(v_1,v_1')$ and $(v_1,v_2')$ are arcs in $G$.  Since $v_1$ has out-degree two, we can assume that $(v_2,v_3')$ is also an arc.  By Lemma  \ref{properties of Type 1 vertices}, we have $v_1 \in N^{-2}(v_3')$.  As the out-neighbours of $v_1$ are $v_1'$ and $v_2'$, it follows that either $(v_1',v_3')$ or $(v_2',v_3')$ is an arc, contradicting Lemma \ref{elements of S' not adjacent}.

Hence we can assume that $N^+(v_i) = \{ v_i',v_i^+\} $ for $i = 1,2,3$ where $v_i^+ \not \in S'$ for $i = 1,2,3$.  By Lemma \ref{forms of outlier sets |S| = 3}, there is an outlier set $\Omega $ contained in $S$.  By Lemma \ref{every v' neighbour of an outlier}, $S'$ must be contained in $N^+(\Omega )$; by inspection this is impossible.
\end{proof}

\begin{lemma}\label{there is a type 1 vertex}
There is a Type 1 vertex.
\end{lemma}
\begin{proof}
Suppose that $N_2 = 3$.  We can set $N^-(v_i') \cap S = S-\{ v_i\} $ for $i = 1,2,3$.  Then for $i \not = j$ we must have $d(v_i',v_j') \geq k$, as $N^-(v_i') \cap N^-(v_j') \not = \varnothing $.  As $N^+(v_3) = \{ v_1',v_2'\} $, it follows that $v_3' \in O(v_3)$.  By Lemma \ref{every v' neighbour of an outlier}, $S' \subseteq N^+(O(v_3))$.  By Lemma \ref{elements of S' not adjacent}, $N^+(v_3') \cap S' = \varnothing $, so, as $G$ has out-degree $d = 2$, this is not possible.
\end{proof}

\begin{lemma}\label{S in-neighbours of Type 2 vertex not an outlier set}
Let $v'$ be a Type 2 vertex.  Then $S \cap N^-(v')$ is not an $\Omega $-set.  Also, every vertex in $G$ can reach exactly one member of $S \cap N^-(v')$ by a $\leq k$-path.  If $v',v'' \in S'$ are both Type 2 vertices, then $S \cap N^-(v') \not = S \cap N^-(v'')$.
\end{lemma}
\begin{proof}
For definiteness, suppose that $v_1'$ is a Type 2 vertex, with $S \cap N^-(v_1') = \{ v_1,v_2\} $.  Suppose that $\{ v_1,v_2\} $ is an $\Omega $-set.  By Lemma  \ref{every v' neighbour of an outlier}, $S' \subseteq N^+\{ v_1,v_2\} $.  We can thus suppose that there are arcs $(v_1,v_2')$ and $(v_2,v_3')$ in $G$.  By Lemma \ref{there is a type 1 vertex} we can assume that $v_2'$ is Type 1.  By Lemma \ref{properties of Type 1 vertices}, $v_2 \in N^{-2}(v_2')$ so that $N^+(v_2) \cap N^-(v_2') = \{ v_1',v_3'\} \cap N^-(v_2') \not = \varnothing $, contradicting Lemma \ref{elements of S' not adjacent}.  Therefore $\{ v_1,v_2\} $ is not an $\Omega $-set.

Let $u$ be a vertex that can reach both $v_1$ and $v_2$ by a $\leq k$-path.  Let $u^- \in N^-(u)$ and $N^+(u^-) = \{ u,u^+\} $.  By Lemma \ref{every v an outlier} we must then have $O(u^+) = \{ v_1,v_2\} $, a contradiction.  Suppose now that $v_1'$ and $v_2'$ are Type 2 vertices and $S \cap N^-(v_1') = S \cap N^-(v_2') = \{ v_1,v_2\}$.  Then $N^+(v_1) = N^+(v_2)$, which by the preceding argument contradicts the Amalgamation Lemma.
\end{proof}

\begin{corollary}\label{there is a Type 1 vertex}
There are two Type 1 vertices and a unique Type 2 vertex.
\end{corollary}
\begin{proof}
Suppose that $v_1'$ and $v_2'$ are Type 2 vertices, so that $v_3'$ is Type 1.  By Lemma \ref{S in-neighbours of Type 2 vertex not an outlier set} we can assume that $S \cap N^-(v_1') = \{ v_1,v_2\} , S \cap N^-(v_2') = \{ v_1,v_3\} $ and $v_2 \in N^-(v_3')$.  By Lemma \ref{properties of Type 1 vertices}, we then have $v_1,v_3 \in N^{-2}(v_3')$.  It follows that $v_1$ has an out-neighbour in $N^-(v_3')$, contradicting Lemma \ref{elements of S' not adjacent}.  
\end{proof}

We can therefore assume for the remainder of this subsection that $v_1'$ and $v_2'$ are Type 1 and $v_3'$ is Type 2.  Write $x$ for the in-neighbour of $v_3'$ that does not lie in $S$.

\begin{lemma}
$S \cap N^-(v_1') = S \cap N^-(v_2')$.
\end{lemma}
\begin{proof}
Suppose that $S \cap N^-(v_1') \not = S \cap N^-(v_2')$.  By Lemma \ref{properties of Type 1 vertices}, without loss of generality we can put

\[ v_1 \in N^-(v_1'), v_2,v_3 \in N^{-2}(v_1'), v_2 \in N^-(v_2'), v_1,v_3 \in N^{-2}(v_2').\] 
We cannot have $N^+(v_1) \subset S'$, or $v_1 \in N^{-2}(v_2')$ would imply that two vertices of $S'$ are adjacent.  Thus $v_1 \not \in N^-(v_3')$.  Similar reasoning applies to $v_2$.  However, there are two members of $S$ in $N^-(v_3')$, a contradiction. 
\end{proof}

We can now set without loss of generality $v_1 \in N^-(v_1') \cap N^-(v_2'), v_2,v_3 \in N^{-2}(v_1') \cap N^{-2}(v_2')$ and $S \cap N^-(v_3') = \{ v_2,v_3\} $.  It follows from Lemma \ref{S in-neighbours of Type 2 vertex not an outlier set} that for every vertex $u$ we have $|O(u) \cap \{ v_2,v_3\} | = 1$.  We can assume that $v_3 \in O(v_1), v_2 \not \in O(v_1)$.  Write $N^+(v_2) = \{ v_3',v_2^+\} $ and $N^+(v_3) = \{ v_3',v_3^+\} $.  By the Amalgamation Lemma $v_2^+ \not = v_3^+$.

\begin{lemma}
$v_3'$ is not an outlier.
\end{lemma}
\begin{proof}
Suppose that for some outlier set we have $v_3' \in \Omega $.  By  Lemma \ref{elements of S' not adjacent}, $N^+(v_3') \cap S' = \varnothing $, so that we cannot have $S' \subseteq N^+(\Omega )$, contradicting Lemma \ref{every v' neighbour of an outlier}. 
\end{proof}

As there is a $\leq k$-path from $v_1$ to $v_2$, either $v_1'$ or $v_2'$ lies in $T^-(v_2)$; assume that $v_1' \in T^-(v_2)$.  Suppose that $d(v_1',v_2) \leq k-2$.  There is a path of length 2 from $v_2$ to $v_1'$, so there would be a $\leq k$-cycle through $v_1'$, which is impossible.  It follows that $d(v_1',v_2) = k-1$, so that $d(v_1',v_3') = k$.  

As $v_1$ must lie in $T_{-k}(v_3')$, we must have $v_2' \in T^-(v_3')$.  If $d(v_2',v_3') \leq k-2$, then there would be two $\leq k$-paths from $v_2$ and $v_3$ to $v_3'$.  Thus $d(v_2',v_3') = k-1$.  If $v_2'$ lies in $N^{-(k-2)}(v_2)$ or $N^{-(k-2)}(v_3)$, there would be a $\leq k$-cycle in $G$ through $v_2$ or $v_3$ respectively.  Hence $v_2' \in N^{-(k-2)}(x)$ and $v_1' \not \in T^-(x)$.

\begin{corollary}
$x$ is not an outlier.
\end{corollary}
\begin{proof}
As $v_2' \in N^{-(k-2)}(x)$ and $v_1',v_3',v_2,v_3 \not \in T^-(x)$, $|T_{-k}(x)| = M(2,k)+2$.
\end{proof}

\begin{lemma}\label{v1 not out-neighbour of v2 or v3}
$v_1 \not \in \{ v_2^+,v_3^+\} $, i.e. $v_1$ is not an out-neighbour of $v_2$ or $v_3$.
\end{lemma}
\begin{proof}
Suppose that $v_1 = v_2^+$.  Denote the in-neighbour of $v_1'$ that does not belong to $\{ v_1,v_3^+\} $ by $v_1^*$ and the in-neighbour of $v_2'$ that does not belong to $\{ v_1,v_3^+\} $ by $v_2^*$.  By Lemma \ref{properties of Type 1 vertices}, $v_1'$ and $v_2'$ are not outliers and $d(v_2',v_1') = d(v_3',v_1') = k$.  We cannot have $v_2' \in T^-(v_1)$, or there would be a $k$-cycle through $v_1$.  Also $v_2' \not \in T^-(v_3^+)$, or there would be a $k$-cycle through $v_3^+$.  Likewise, $v_3' \not \in T^-(v_1)$, or there would be a $(k-1)$-cycle through $v_2$, and $v_3' \not \in T^-(v_3^+)$, or there would be two $\leq k$-paths from $v_3$ to $v_3^+$.  It follows that $v_2',v_3' \in N^{-(k-1)}(v_1^*)$ and likewise we have $v_1',v_3' \in N^{-(k-1)}(v_2^*)$.  As $S \cap T^-(v_1^*) = S \cap T^-(v_2^*) = \varnothing $, it follows that $|T_{-k}(v_1^*)| = |T_{-k}(v_2^*)| = M(2,k)+2$, so that $O^-(v_1^*) = O^-(v_2^*) = \varnothing $.  By Lemma \ref{every v' neighbour of an outlier}, possible $\Omega $-sets are 
\[ \{ v_2,v_1\} , \{ v_2,v_3^+\} , \{ v_3,v_1\} , \{ v_3,v_3^+\} .\]   
But then every $\Omega $-set contains either $v_1$ or $v_3^+$ and $N^+(v_1) = N^+(v_3^+)$, contradicting the Amalgamation Lemma.
\end{proof}

\begin{theorem}
There are no $(2,k,+2)$-digraphs with in-degree sequence $(1,1,1,2,\dots ,2,3,3,3)$ for $k \geq 3$.
\end{theorem}
\begin{proof}
By Lemma \ref{v1 not out-neighbour of v2 or v3}, $N^-( v_1') = N^-(v_2') = \{ v_1,v_2^+,v_3^+\} $.  By Lemma \ref{properties of Type 1 vertices}, we have $v_2' \in N^{-k}(v_1')$.  But it is easy to see that whether $v_2'$ lies in $T^-(v_1), T^-(v_2^+)$ or $T^-(v_3^+)$, there will be a $k$-cycle through $v_1$, $v_2^+$ or $v_3^+$ respectively.
\end{proof}

It remains only to deal with the case $k = 2$.  First we need to prove the equivalent of Lemma \ref{elements of S' not adjacent}, i.e. that $S'$ is an independent set.

\begin{lemma}\label{S' independent k 2}
For $k = 2$, no two members of $S'$ are adjacent.
\end{lemma}
\begin{proof}
Suppose that $v_2' \rightarrow v_1'$.  By $2$-geodecity, $v_1'$ must be Type 2, $v_2'$ is Type 1 and $O^-(v_1') = \varnothing $.  By the same reasoning, if $v_3' \rightarrow v_2'$, then $v_2'$ would be Type 2, a contradiction.  Hence we can assume that $N^-(v_1') = \{ v_1,v_2,v_2'\} , N^-(v_2') = \{ v_3,x,y\} , N^-(v_1) = \{ z\} $ and $ N^-(v_2) = \{ v_3'\} $, where $d^-(x) = d^-(y) = d^-(z) = 2$. 

Obviously $v_2 \not \rightarrow v_3'$, so $v_3'$ is either Type 1 or Type 2.  Suppose that $v_3 \rightarrow v_3'$.  Then $v_2'$ cannot be adjacent to $v_3'$, or there would be two $\leq 2$-paths from $v_3$ to $v_3'$.  Hence $v_2' \rightarrow z$.  This implies that $x$ and $y$ are not adjacent to $z$, as they can already reach $z$ via $v_2'$.  Therefore $x$ and $y$ are adjacent to $v_3'$.  Now $v_3'$ cannot be adjacent to any of $v_3, x$ or $y$, so $v_3' \rightarrow z$, thereby creating paths $v_3 \rightarrow v_3' \rightarrow z$ and $v_3 \rightarrow v_2' \rightarrow z$.  Alternatively, one can see that $N^-(v_2') = N^-(v_3')$, which is impossible, since $|T_{-2}(v_2')| = 9$.  Therefore $v_3 \not \rightarrow v_3'$.  Applying the same approach to $x$ and $y$, we see that these vertices also have no arcs to $v_3'$.  Hence all of $v_3, x$ and $y$ are adjacent to $z$.  However, as $d^-(z) = 2$, this is not possible.
\end{proof}

\begin{lemma}
Every vertex in $S'$ is Type 2.
\end{lemma}
\begin{proof}

Suppose that $S'$ contains a Type 1 vertex; say $v_1'$ is Type 1, with $N^-(v_1') = \{ v_1,x,y\} $, where $d^-(x) = d^-(y) = 2$.  Write $N^-(v_1) = \{ z\} , N^-(x) = \{ x_1,x_2\} , N^-(y) = \{ y_1,y_2\} $.  

Note that we cannot have $|N^-(x)\cap S'|=2$ or $|N^-(y)\cap S'| =2$.  For suppose that $y_1 = v_2',y_2=v_3'$.  $v_1'$ is not an in-neighbour of either of these vertices by Lemma \ref{S' independent k 2}.  By $2$-geodecity no in-neighbourhood can contain both end-points of an arc, so the in-neighbourhoods of $v_2'$ and $v_3'$ must consist of one vertex from $\{ z,v_1\} $ and both of $x_1,x_2$.  However, $x_1$ and $x_2$ have out-degree two, so this is not possible.  The same argument shows that we cannot have $|N^-(x)\cap S'| = |N^-(y)\cap S'| = 1$, for then $v_2'$ and $v_3'$ would have to be adjacent, in violation of Lemma \ref{S' independent k 2}.  There are thus two possibilities up to isomorphism: i) $v_2' = z, v_3' = x_1, v_2 = x_2, v_3 = y_1$ or ii) $v_2' = z, v_3' = x_1, v_2 = y_1, v_3 = y_2$.

In case i), as $S'$ is independent we must have $N^-(v_3') = \{ v_1,v_3,y_2\} $.  However, no arc from $v_3'$ can be inserted to $N^{-2}(v_1')$ without violating either $2$-geodecity or Lemma \ref{S' independent k 2}.

In case ii), we must have $N^-(v_3') = S$, so that $v_2'$ cannot have any in-neighbours in $S$, contradicting Lemma \ref{in-neighbours in S}.  Therefore $S'$ contains no Type 1 vertices.  From $N_1 + 2N_2 + 3N_3 \leq 6$, it now follows that every vertex of $S'$ is Type 2 and $N^+(v_i) \subset S'$ for $i = 1,2,3$.
\end{proof}

Distinct vertices from $S$ cannot have identical out-neighbourhoods; for example, if $N^+(v_1) = N^+(v_2) = \{ v_1',v_2'\} $, then $v_3'$ could not be Type 2.  For $i = 1,2,3$, we can therefore set $N^-(v_i') = (S - \{ v_i\} ) \cup \{ x_i\} $, where $d^-(x_i) = 2$.  As $S' \cap N^-(S') = \varnothing $, we see that $N^+(v_i)\cap T^-(v_i') = \varnothing $ for $i = 1,2,3$, so that $O^-(v_i') = \{ v_i\} $ for $i = 1,2,3$.  We now have enough information to complete the proof.

\begin{theorem}
There are no $(2,2,+2)$-digraphs with in-degree sequence $(1,1,1,2,2,2,3,3,3)$.
\end{theorem}
\begin{proof}
Write $N^-(v_i) = \{ z_i\} $ for $i = 1,2,3$ and put $N^-(x_1) = \{ y_1,y_2\} $.  There are three distinct cases to consider, depending on the position of $v_2'$ and $v_3'$ in $T_{-2}(v_1')$: i) $v_2' = z_2, v_3' = z_3$, ii) $v_2' = z_2, v_3' = y_1$ and iii) $v_2' = y_1, v_3' = y_2$.

Consider case i).  $v_1'$ is adjacent to neither $v_2'$ nor $v_3'$ by Lemma \ref{S' independent k 2} and cannot be adjacent to both elements of $N^-(x_1)$ by $2$-geodecity.  Hence we can assume that $N^+(v_1') = \{ v_1, y_2\} $.  Hence $v_1'$ has paths of length two to $v_2'$ and $v_3'$ via $v_1$.  It follows that $y_2$ cannot be adjacent to any of $v_1, v_2', v_3'$ or $y_1$ without violating $2$-geodecity.  

In case ii), the only vertex other than $v_1$ and $v_2$ that can be an in-neighbour of $v_3'$ is $z_3$, but in this case $v_3'$ cannot be adjacent to any of $y_2, z_3, v_1$ or $v_2'$, so we have a contradiction.  Finally, in case iii) there are two $2$-paths from $v_1$ to $x_1$. 
\end{proof}

\section{Degree sequence $(1,1,1,1,2,\dots ,2,3,3,3,3)$}

We turn to our final in-degree sequence.  In this case the abundance of elements in $S$ and $S'$ enables us to easily classify all $\Omega $-sets of $G$.  A parity argument based on the number of occurrences of the outlier sets then allows us to obtain a contradiction.
\begin{lemma}\label{|S| = 4 lemma}\label{nice outlier properties}
For every vertex $u$, $O(N^+(u)) = S$, $N^+(O(u)) = S'$ and $O(u) \subset S$.  If $\Omega $ is an outlier set, so is $S - \Omega $.
\end{lemma}
\begin{proof}
By Lemmas \ref{every v an outlier} and \ref{every v' neighbour of an outlier} we have $S \subseteq O(N^+(u))$ and $S' \subseteq N^+(O(u))$.  As $|S| = |S'| = 4$, we must have equality in the inclusions.  If $O(u) = \Omega $, let $u^-, u^+$ be such that $N^+(u^-) = \{ u,u^+\} $; then we must have $O(u) \cup O(u^+) = S$, so that $O(u) \subset S$ and $O(u^+) = S - \Omega $.
\end{proof}

\begin{corollary}\label{|S| = |S'| = 4 2 in-neighbours in S}
For all $v' \in S'$, $|N^-(v') \cap S| = 2$.
\end{corollary}
\begin{proof}
Let $N^+(v') = \{ w_1,w_2\} $.  Let $O(w_1) = \Omega _1 , O(w_2) = \Omega _2 $, where $\Omega _1 \cup \Omega _2 = S$.  Write $N^+(w_1) = \{ w_3,w_4\} $ and $N^+(w_2) = \{ w_5,w_6\} $.  By $k$-geodecity, at most one in-neighbour of $v'$ lies in $T(w_3)$ and at most one lies in $T(w_4)$ and furthermore $w_1 \not \in N^-(v')$.  It follows that an in-neighbour of $v'$ lies in $\Omega _1 $.  Applying the argument to $w_2$, another in-neighbour of $v'$ lies in $\Omega _2 $.  Hence $|N^-(v') \cap S| \geq 2$ for all $v' \in S'$.

Suppose that $|N^-(v_1') \cap S| = 3$.  As $|N^-(v_i') \cap S| \geq 2$ for $i = 2,3,4$, we must have $\sum _{i = 1}^4 d^+(v_i) \geq 9$, which is impossible.  
\end{proof}

\begin{lemma}\label{neighbourhoods of S distinct}
No two elements of $S$ have the same out-neighbourhood.
\end{lemma}
\begin{proof}
Suppose that $V \subset S$, $|V| = 2$ and $|N^+(V)| = 2$.  By Lemma \ref{nice outlier properties}, $V$ is not an $\Omega $-set, as $N^+(V) \not = S'$.  Suppose that there exists a vertex $u$ that can reach both vertices of $V$ by $\leq k$-paths.  Then by Lemma \ref{nice outlier properties} $O(u) = S - V$, so that $S-(S-V) = V$ must be an $\Omega $-set, a contradiction.  Now we have a pair of vertices with identical out-neighbourhoods and with non-empty intersection with every $\Omega $-set, violating the Amalgamation Lemma.    
\end{proof}

\begin{lemma}\label{only 2 outlier sets}
There are only two distinct $\Omega $-sets.
\end{lemma}
\begin{proof}
Let $u \in V(G)$ and $N^+(u) = \{ u_1,u_2\} $ and write $O(u_1) = \Omega _1 , O(u_2) = \Omega _2$, where $\Omega _1 \cup \Omega _2 = S$.  By Lemma \ref{neighbourhoods of S distinct}, for $1 \leq i,j \leq 4$ and $i \not = j$

\[ N^-(v_i') \cap S \not = N^-(v_j') \cap S.\]
None of the sets $N^-(v_i') \cap S$ can be $\Omega $-sets, since any such set has at most three out-neighbours.  There are $4 \choose 2$ two-element subsets of $S$, all of which are accounted for by the two outlier sets $\Omega _1 $ and $\Omega _2 $ and the four sets $N^-(v_i') \cap S$.

\end{proof}

\begin{theorem}
There are no $(2,k,+2)$-digraphs with in-degree sequence $(1,1,1,1,2,\dots ,2,3,3,3,3)$ for $k \geq 2$.
\end{theorem}
\begin{proof}
Let the distinct outlier sets of $G$ be $\Omega _1$ and $\Omega _2$.  As $G$ has odd order $2^{k+1}+1$, one of these sets must occur more frequently as an $\Omega $-set than the other.  Take an arbitrary vertex $u$ with $O(u) = \Omega _1 $ and consider $T_k(u) \cup \Omega _1$, which contains all vertices of $G$ without repetitions.  By Lemmas \ref{|S| = 4 lemma} and \ref{only 2 outlier sets}, for every vertex $w$ of $G$ with out-neighbours $w_1,w_2$, we have $O(w_1) = \Omega _1 , O(w_2) = \Omega _2$ or vice versa, so half of the vertices in $T_k(u) - \{ u\} $ have outlier set $\Omega _1 $ and half have outlier set $ \Omega _2 $.  As $O(u) = \Omega _1 $ and each element of $\Omega_1$ has outlier-set $\Omega _2 $, it follows that the set $\Omega _2 $ occurs $2^k + 1$ times as an $\Omega $-set and $\Omega _1 $ occurs $2^k$ times.  However, repeating the argument with a vertex $u$ with $O(u) = \Omega _2 $ leads to the opposite conclusion, a contradiction.
\end{proof}

This concludes the proof of the main theorem.

\section{Extremal non-diregular $(d,k,+\epsilon )$-digraphs}

We have seen that there are no non-diregular $(2,2,+\epsilon )$-digraphs for $\epsilon \leq 2$ \cite{MirSil}; however, in \cite{Tui} it is proven that there exist two distinct diregular $(2,2,+2)$-digraphs up to isomorphism.  It is therefore of interest to determine the smallest possible excess of a non-diregular $(d,k,+\epsilon )$-digraph for $d = k = 2$ and other values of $d$ and $k$.  In \cite{MilSla2} and \cite{SlaMil} it is shown that from a diregular digraph of order $n$, maximum out-degree $d$ and diameter $k$  that contains a pair of vertices with identical out-neighbourhoods there can be derived a non-diregular digraph of order $n-1$, maximum out-degree $d$ and diameter $\leq k$ by means of a `vertex deletion scheme'.  By these means large non-diregular digraphs are constructed from Kautz digraphs in \cite{SlaMil}.  We now describe a `vertex-splitting' construction that enables us to derive a non-diregular $(d,k,+(\epsilon +1))$-digraph from a $(d,k,+\epsilon )$-digraph.

\begin{theorem}[Vertex-splitting construction]\label{vertex insertion scheme}
If there exists a $(d,k,+\epsilon )$-digraph, then for any $0 \leq r \leq d$ there also exists a non-diregular $(d,k,+(\epsilon +1))$-digraph with minimum in-degree $\leq d-r$.
\end{theorem}
\begin{proof}
Let $G$ be a $(d,k,+\epsilon )$-digraph and choose a vertex $u$ with in-degree $\geq d$.  Form a new digraph $G'$ by adding a new vertex $w$ to $G$, setting $N^+(w) = N^+(u)$ and redirecting $d-r$ arcs that are incident to $u$ to be incident to $w$.  Colloquially, the vertex $u$ is split into two vertices.  $G'$ is easily seen to also be $k$-geodetic with minimum out-degree $\geq d$.
\end{proof}

We call a $(d,k,+\epsilon )$-digraph with smallest possible excess a $(d,k)$-geodetic cage.  It follows from Theorem \ref{vertex insertion scheme} that the order of a smallest possible non-diregular $k$-geodetic digraph with minimum out-degree $\geq d$ exceeds the order of a $(d,k)$-geodetic cage by at most one.  In particular, as all $(2,2)$-geodetic cages are diregular with order nine \cite{Tui}, smallest possible non-diregular $2$-geodetic digraphs with minimum out-degree $\geq 2$ have order ten.  It would be of great interest to determine whether or not there exist $(2,k,+3)$-digraphs for $k \geq 3$ in both the diregular and non-diregular cases.  

Experience shows that non-diregularity of a digraph with order close to the Moore bound makes $k$-geodecity difficult to satisfy.  This leads us to make the following two conjectures.
\begin{conjecture}
	All geodetic cages are diregular.
\end{conjecture}
\begin{conjecture}
	All smallest possible non-diregular $(d,k,+\epsilon )$-digraphs can be derived from a diregular $(d,k)$-geodetic cage by the vertex splitting construction.
\end{conjecture}

\textbf{Acknowledgements}
\newline
The author thanks the three anonymous referees, whose careful reading and considered comments helped to improve the article.

\end{document}